\documentclass[12pt]{article}
\usepackage{amsmath,amssymb,amstext,dsfont,fancyvrb,float,fontenc,graphicx,subfigure, theorem}

\usepackage[letterpaper]{geometry}
\ifx\volno\undefined\def\volno{1}\fi
\ifx\volyear\undefined\def\volyear{2017}\fi
\ifx\papno\undefined\def\papno{P0.0}\fi

\newfont{\footsc}{cmcsc10 at 8truept}
\newfont{\footbf}{cmbx10 at 8truept}
\newfont{\footrm}{cmr10 at 10truept}

\usepackage{fancyhdr}
\usepackage{relsize}
\usepackage{sectsty}
\allsectionsfont{\larger[-1]} 

\renewcommand\paragraph{\@startsection{paragraph}{4}{\z@}
                                    {2ex \@plus.5ex \@minus.2ex}
                                    {-1em}
                                    {\normalfont\normalsize\bfseries}}

\renewcommand\subparagraph{\@startsection{subparagraph}{5}{\parindent}
                                       {2ex \@plus.5ex \@minus .2ex}
                                       {-1em}
                                      {\normalfont\normalsize\bfseries}}

\newlength{\BiblioSpacing}
\renewenvironment{thebibliography}[1]{
\begin{oldthebibliography}{#1}
\setlength{\parskip}{\BiblioSpacing}
\setlength{\itemsep}{\BiblioSpacing}
}
{
\end{oldthebibliography}
}

\usepackage[strict]{changepage}
\def\abstractname{Abstract -}   
\def\abstract{\begin{adjustwidth}{1cm}{1cm} \par    \footnotesize \noindent {\bf \abstractname} 
\def\endabstract{ \end{adjustwidth} \smallskip }}

{\theorembodyfont{\itshape}\newtheorem{theorem}{Theorem}[section]}
{\theorembodyfont{\itshape}\newtheorem{proposition}[theorem]{Proposition}}
{\theorembodyfont{\itshape}\newtheorem{definition}[theorem]{Definition}}
{\theorembodyfont{\itshape}}
{\theorembodyfont{\itshape}}
{\theorembodyfont{\rm}}
{\theorembodyfont{\rm}\newtheorem{remark}[theorem]{Remark}}
{\theorembodyfont{\rm}}
{\theorembodyfont{\rm }}

\title{\Large\bf Partition Problems and a Pattern of Vertical Sums}
\author{\sc E. Krinsky, S. Raianu, and A. Wittmond}
\date{}
\begin{document}
\maketitle
\thispagestyle{fancy}

\vskip 1.5em

\begin{abstract}
We give a possible explanation for the mystery of a missing number in the statement of a problem that asks for the non-negative integers to be partitioned into three subsets. We interpret the missing number as one of the clues that can lead to a more standard solution to the problem, using only congruence modulo five, and we give the details to the new solution, which is based on an algorithm inspired by noticing alternating differences between sums of elements of the same rank in the three sets. Our new  solution is equivalent to the partition consisting of numbers with remainders one or three modulo five, two or four modulo five, and multiples of five, which we call the standard partition. We then find all other similar statements with the same pattern of sums, we apply the algorithm to them, and we describe all the partitions obtained, up to a certain equivalence. There are $279936$ different such statements, they produce twenty different partitions (other than the standard one) whose sets of the first five columns are not permutations of each other, and only one of them (the one produced by the original statement of the problem we study) is equivalent to the standard partition. Finally, we construct infinitely many partitions equivalent to the standard one, and we give a possible generalization and a sample partition problem asking for the non-negative integers to be partitioned into four sets.
\end{abstract}
 
\begin{keywords}
Congruence modulo an integer; partitions of sets; partitions of integers
\end{keywords}

\begin{MSC}
11A07; 11P81; 05A15; 05A17; 05A18
\end{MSC}

\section{The problem and its two solutions}\label{sect1}

In this paper we study Problem 2.1.18 of \cite{zeitz}, whose statement is given in Figure 1 on the next page.

\begin{figure}[!htb]
  \begin{center}
     \includegraphics[width=100mm,scale=2.50]{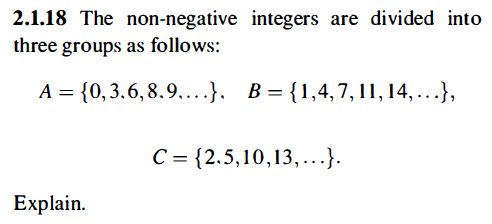}
  \end{center}
  \caption{\label{fig:problem} Scan from \cite[p. 24]{zeitz}}
\end{figure}

We note first that there are infinitely many ways to complete the partition required by the problem. Perhaps the easiest one is to take the complement of the given set of numbers, order it increasingly, then use it to fill in all the empty positions in some order. But these solutions do not use any potential hints, and can be applied to any similar statement involving arbitrary given numbers.

The solution given in the Instructor's Guide for \cite{zeitz} is the following: ``The first set contains only numbers whose numerals are drawn exclusively with curves; the second contains ``linear" numerals; the third contains the mixed numerals." A quick inspection of the statement shows that 12 does not 
appear, while 13 and 14 do. According to the solution, 12 should belong to $C$,
and so we would expect $C$ to be given as 
$C=\{2,5,10,12,13,\ldots\}$, or at least as $C=\{2,5,10,12,\ldots\}$. The other thing that looks like a clue is that there are five elements given in $A$ and $B$, but only four in $C$, and this looks like a hint that we should first find the fifth element of $C$. In this section
we show that these two hints (the missing 12 and the number of elements of the three sets) lead to another solution to the problem, and that solution can be found (starting with the seventh elements of each set) uing a standard
congruence classes argument. We begin by writing the sets 
$A$, $B$ and $C$ one below the other\\
\[ \begin{array}{cccccccc}
A & = & \{  0, & 3, & 6, & 8, & 9, & \ldots  \}\\
B & = & \{  1, & 4, & 7, & 11, & 14, & \ldots  \}\\
C & = & \{  2, & 5, & 10, & 13, & & \ldots  \},\\
\end{array} \]

We see that in our font 7 looks like it should belong to $C$. However, on a microwave oven or a cable box display, all numbers are ``linear". The solution we find in this section does not depend on fonts.

Let's see what happens when we add the first elements in each set, then the second elements, and so on. The first four sums of elements of the same rank are: 3, 12, 23, and 32. We notice that the 
differences between consecutive sums are 9, 11, 9. \\

{\bf We will assume that this is a pattern, and we will find another solution to the problem in Figure \ref{fig:problem} based on this assumption.}\\

We denote, as usual, by $\lfloor x\rfloor$ the greatest integer less than or equal to $x$, and we give a couple of definitions.

\begin{definition}\label{stanpar} The {\em\bf standard partition} of the non-negative integers is given by
\begin{align*}
S_1 &=\{3n-\lfloor\frac{n}{2}\rfloor-2\mid n\ge 1\}\\
S_2 & =\{3n-\lfloor\frac{n}{2}\rfloor-1\mid n\ge 1\}\\
S_3 &= \{5n-5\mid n\ge 1\}
\end{align*}
Note that the sum of the $n$-th elements is $S(n)=11n-2\lfloor\frac{n}{2}\rfloor-8$. 
\end{definition}

\begin{remark}\label{rempar}
The formulas for the three consecutive elements of ranks $2k+1$, $2k+2$, and $2k+3$ ($k\ge 0$) in the standard partition can be also written as:\\
 \[ \begin{array}{ccc}
5k+1 & 5k+3 & 5k+6 \\
5k+2 & 5k+4 & 5k+7 \\
5\cdot(2k) & 5\cdot(2k+1) & 5\cdot(2k+2),
\end{array} \]
This shows that $S_1$ consists of the numbers with remainder 1 or 3 (mod 5), 
the elements of $S_2$ are congruent to 2 or 4 (mod 5), while $S_3$ is the set of
multiples of 5.
It also shows that the standard partition also has alternating differences between vertical sums of 9 and 11: $S(2k+2)-S(2k+1)=(3-1)+(4-2)+5=9$, while $S(2k+3)-S(2k+2)=(6-3)+(7-4)+5=11$ for $k\ge 0$. (This can also be written as $S(n+1)-S(n)=10+(-1)^n$ for all $n\ge 1$). 
\end{remark}

\begin{definition}\label{eqpar}
Two partitions of the non-negative integers, given by $A=\{a_k\mid k\ge 1\}$, $B=\{b_k\mid k\ge 1\}$, $C=\{c_k\mid k\ge 1\}$, and $D=\{d_k\mid k\ge 1\}$, $E=\{e_k\mid k\ge 1\}$, $F=\{f_k\mid k\ge 1\}$, are said to be {\em\bf equivalent} if there exists a number $N$ such that for all $k>N$ we have $a_k=d_k$, $b_k=e_k$, and $c_k=f_k$ (note that this implies that the sets $\{a_k,b_k,c_k\mid 1\le k\le N\}$ and $\{d_k,e_k,f_k\mid 1\le k\le N\}$ are permutations of each other).
\end{definition}

With these two definitions, we prove the following:
\begin{proposition}\label{secsol}
There exists a solution to the problem from  Figure \ref{fig:problem}  which is equivalent to the standard partition.
\end{proposition}
\begin{proof}
We assume that the pattern of the vertical sums continues, so that
the fifth sum is $32+11=43$, and so the fifth number in $C$ is 
$43-9-14=20$. The sixth sum is then $43+9=52$, and we try the first two 
unused numbers,
12 and 15, as the sixth elements of $A$ and $B$, respectively. We get that the
sixth element in $C$ is $52-12-15=25$. The next two
unused numbers are then 16 and 
17. Putting them on the seventh position in $A$ and $B$, and assuming the 
seventh sum is $52+11=63$, we get that the seventh element of $C$ is $63-16-17=30$.
From this point on ($n\ge 8$) we use the following\\[1mm]
{\bf Algorithm: We put the first two
unused numbers as the $n$-th elements of $A$ and $B$, then add them and substract their sum from $S(n)$ to find the $n$-th element of $C$.}\\[1mm]
In view of Remark \ref{rempar}, the standard partition can be constructed using this algorithm for $n\ge 2$.
Now we take $N=6$ in Definition \ref{eqpar}: the set of the first six elements in the solution constructed above is a permutation of the set of the first six elements in the standard partition. Hence, from the seventh element on, the solution constructed above and the standard partition have the same elements, because they are both constructed using the same algorithm and the first two unused numbers are the same at each step. 
\end{proof}

We end this section by showing that we can start with the standard partition and change it by a permutation in order to get the statement in Figure \ref{fig:problem} (this can be seen as another proof of Proposition \ref{secsol}). In the next section we will show other ways of obtaining partitions equivalent to the standard partition by performing certain permutations. We start now with $A$, $B$, and $C$ the standard partition: $A$ is the set of numbers congruent to 1 or 3 
(mod 5), $B$ consists of the numbers congruent to 2 or 4 (mod 5), while $C$ is
the set of multiples of 5:
\[ \begin{array}{ccccccccc}
A & = & \{  1, & 3, & 6, & 8, & 11, & 13, &\ldots  \}\\
B & = & \{  2, & 4, & 7, & 9, & 12, & 14, &\ldots  \}\\
C & = & \{  0, & 5, & 10, & 15, & 20, & 25, &\ldots  \}.\\
\end{array} \]
After noticing that if we restrict our attention to the second, third and 
fourth columns only, the numbers on the first row can be written using curved
lines only, the numbers on the third row need both curves and straight lines,
while 4 and 7 can be drawn with straight lines only (remember to keep Fig. \ref{fig:problem} in mind when we talk about 7), we permute the numbers on
the first column accordingly, and we get
\[ \begin{array}{ccccccccc}
A & = & \{  0, & 3, & 6, & 8, & 11, & 13, &\ldots  \}\\
B & = & \{  1, & 4, & 7, & 9, & 12, & 14, &\ldots  \}\\
C & = & \{  2, & 5, & 10, & 15, & 20, & 25, &\ldots  \}.\\
\end{array} \]
We now trade 9 in the second set for the first available ``linear''
number, which is 11, and we get\\
\[ \begin{array}{ccccccccc}
A & = & \{  0, & 3, & 6, & 8, & 9, & 13, &\ldots  \}\\
B & = & \{  1, & 4, & 7, & 11, & 12, & 14, &\ldots  \}\\
C & = & \{  2, & 5, & 10, & 15, & 20, & 25, &\ldots  \}.\\
\end{array} \]
Now the fourth sum is broken, so we must trade 15 for 13 to repair it:\\
\[ \begin{array}{ccccccccc}
A & = & \{  0, & 3, & 6, & 8, & 9, & 15, &\ldots  \}\\
\label{par}B & = & \{  1, & 4, & 7, & 11, & 12, & 14, &\ldots  \}\\
C & = & \{  2, & 5, & 10, & 13, & 20, & 25, &\ldots  \}.\\
\end{array} \]
The fourth sum is now repaired, but both the fifth and the sixth ones are broken,
the fifth one is down two, and the sixth one is up two. Consequently, they can
be both repaired by switching the 12 and the 14:\\
\[ \begin{array}{ccccccccc}
A & = & \{  0, & 3, & 6, & 8, & 9, & 15, &\ldots  \}\\
B & = & \{  1, & 4, & 7, & 11, & 14, & 12, &\ldots  \}\\
C & = & \{  2, & 5, & 10, & 13, & 20, & 25, &\ldots  \}.\\
\end{array} \]
All we need to do now is erase 15, 12, 20 and 25 and ask the question from
Fig. \ref{fig:problem}.
\section{Related partition problems}\label{sect2}
After completing the work on the first section we were told by Paul Zeitz \cite{z2} that the initial omission of 12 may have been just a typo, and therefore the pattern of the alternating differences between sums of elements of the same rank in the three sets may have been a fluke. In this section we investigate what the other possible statements are for the problem in Section \ref{sect1} if we preserve the pattern of the sums, and we answer the following questions:
\begin{enumerate}
\item In how many ways can we fill in the 14 given elements of the sets $A$, $B$, and $C$ in Section \ref{sect1}, such that the first four sums are the same, i.e. 3, 12, 23, and 32, and the fifth sum can be 43?
\item For how many of these the algorithm described in the proof of Proposition \ref{secsol} produces a partition that is equivalent to the standard partition?
\item How many different partitions (up to equivalence) are there?
\end{enumerate}
We start answering the first question by showing that there are only 36 different possibilities for the first five elements of $A$, $B$, and $C$, such that the sums are 3, 12, 23, 32, and 43, if we order the $k$-th elements of $A$, $B$, and $C$ increasingly for $k=1,2,3,4,5$. Note that counting the combinations of 15 elements instead of 14 cuts the number of possible statements by two thirds, while the number of non-equivalent
solutions remains the same. The total number of different statements, i.e. the answer to the first question above, is then easily seen to be $36\cdot(3!)^5=6^7=279936$.

In order to list the 36 possibilities described above we notice first that there are only two ways to write 23 as a sum of three numbers greater than or equal to 6: $6+7+10$, and $6+8+9$. There are six ways to write 32 as a sum of three numbers greater than or equal to 7, such that the set of three numbers is disjoint from one of the sets $\{6,7,10\}$ and $\{6,8,9\}$: $7+10+15$, $7+11+14$, $7+12+13$, $8+9+15$, $8+11+13$, and $9+11+12$. The elements of ranks 3 and 4 in the three sets can therefore be one of the following six combinations:
\[ \begin{array}{ccccccccccccccccc}
A  =  \{  6, & 8\} &&\{  6, & 8\} &&\{  6, & 9\} &&\{  6, & 7\}&&\{  6, & 7\}&&\{  6, & 7\}\\
B  =  \{  7, & 9\} &&\{  7, & 11\} &&\{  7, & 11\} &&\{  8, & 10\}&&\{  8, & 11\}&&\{  8, & 12\}\\
C  =  \{  10, & 15\} &&\{  10, & 13\} &&\{  10, & 12\}&&\{  9, & 15\}&&\{  9, & 14\}&&\{  9, & 13\}\\
\end{array} \]
Finally, there are 25 ways to write 43 as a sum of three numbers greater than or equal to 8, such that the set of three numbers is disjoint from at least one of the 6 sets of 6 numbers listed above: $8+13+22$, $8+14+21$, 
$8+15+20$, 
$\ldots$ $13+14+16$. 

We list all of them as columns:
\[ \begin{array}{ccccccccccccccccc}
8&8&8&8&8&9&9&9&9&9&9&10&10&10&10&10&10\\
13&14&15&16&17&11&12&13&14&15&16&11&12&13&14&15&16\\
22&21&20&19&18&23&22&21&20&19&18&22&21&20&19&18&17
\end{array} \]
\[ \begin{array}{cccccccc}
11&11&11&11&12&12&12&13\\
12&13&14&15&13&14&15&14\\
20&19&18&17&18&17&16&16
\end{array} \]

We list now the 36 possibilities. Since for the first two elements in each set the choice is unique, we will only list elements of ranks 3, 4, and 5. Moreover, in each group of six choices, the elements of ranks 3 and 4 are the same, so we only list them once. We assign numbers to these 36 choices from left to right and up to down so the first group contains choices 1 through 6, the second one 7 through 12, and so on. The number assigned to each choice is listed in boldface on top of the elements of rank 5 in that choice. Here are the 36 choices:


\[ \begin{array}{cccccccccc}
&&&&{\bf 1}&{\bf 2}&{\bf 3}&{\bf 4}&{\bf 5}&{\bf 6}\\
A  = & \{  6, & 8,&& 11\} & 11\}& 11\}& 12\}& 12\}& 13\}\\
B  = & \{  7, & 9,&& 12\} & 13\}& 14\}& 13\}& 14\}& 14\}\\
C  = & \{  10, & 15,&& 20\} & 19\}& 18\}& 18\}& 17\}& 16\}\\
\end{array} \]
\[ \begin{array}{cccccccccc}
&&&&{\bf 7}&{\bf 8}&{\bf 9}&{\bf 10}&{\bf 11}&{\bf 12}\\
A  = & \{  6, & 8,&& 9\} & 9\}& 9\}& 9\}& 12\}& 12\}\\
B  = & \{  7, & 11,&& 12\} & 14\}& 15\}& 16\}& 14\}& 15\}\\
C  = & \{  10, & 13,&& 22\} & 20\}& 19\}& 18\}& 17\}& 16\}\\
\end{array} \]
\[ \begin{array}{cccccccccc}
&&&&{\bf 13}&{\bf 14}&{\bf 15}&{\bf 16}&{\bf 17}&{\bf 18}\\
A  = & \{  6, & 9,&& 8\} & 8\}& 8\}& 8\}& 8\}& 13\}\\
B  = & \{  7, & 11,&& 13\} & 14\}& 15\}& 16\}& 17\}& 14\}\\
C  = & \{  10, & 12,&& 22\} & 21\}& 20\}& 19\}& 18\}& 16\}\\
\end{array} \]
\[ \begin{array}{cccccccccc}
&&&&{\bf 19}&{\bf 20}&{\bf 21}&{\bf 22}&{\bf 23}&{\bf 24}\\
A  = & \{  6, & 7,&& 11\} & 11\}& 11\}& 12\}& 12\}& 13\}\\
B  = & \{  8, & 10,&& 12\} & 13\}& 14\}& 13\}& 14\}& 14\}\\
C  = & \{  9, & 15,&& 20\} & 19\}& 18\}& 18\}& 17\}& 16\}\\
\end{array} \]
\[ \begin{array}{cccccccccc}
&&&&{\bf 25}&{\bf 26}&{\bf 27}&{\bf 28}&{\bf 29}&{\bf 30}\\
A  = & \{  6, & 7,&& 10\} & 10\}& 10\}& 10\}& 12\}& 12\}\\
B  = & \{  8, & 11,&& 12\} & 13\}& 15\}& 16\}& 13\}& 15\}\\
C  = & \{  9, & 14,&& 21\} & 20\}& 18\}& 17\}& 18\}& 16\}\\
\end{array} \]
\[ \begin{array}{cccccccccc}
&&&&{\bf 31}&{\bf 32}&{\bf 33}&{\bf 34}&{\bf 35}&{\bf 36}\\
A  = & \{  6, & 7,&& 10\} & 10\}& 10\}& 10\}& 11\}& 11\}\\
B  = & \{  8, & 12,&& 11\} & 14\}& 15\}& 16\}& 14\}& 15\}\\
C  = & \{  9, & 13,&& 22\} & 19\}& 18\}& 17\}& 18\}& 17\}\\
\end{array} \]
 If we eliminate the choices whose sets of first five elements are a permutation of each other (this eliminates the whole fourth group, choices 19 through 24, because its first two columns are a permutation of the first two columns in the first group, then keep 2 and discard 9, keep 3 and discard 27, keep 4 and discard 33, keep 7 and discard 13 and 31, keep 8 and discard 26, keep 14 and discard 25, keep 29 and discard 35) and we also leave out the standard partition (choices 1 and 15), we end up with 20 possibilities. 

Recall that the algorithm described in the proof of Proposition \ref{secsol} puts the first two unused numbers as the next elements of $A$ and $B$, then uses the fact that the sum of the $n$-th elements is $S(n)=11n-2\lfloor\frac{n}{2}\rfloor-8$ to find the next element of $C$. 


\newpage 

\begin{figure}[!htb]
  \begin{center}
     \includegraphics[width=156mm,scale=1.50]{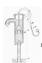}
  \end{center}
  \caption{\label{fig:hask1} The Haskell program, part 1}
\end{figure}

 \newpage

\begin{figure}[!htb]
  \begin{center}
     \includegraphics[width=156mm,scale=1.50]{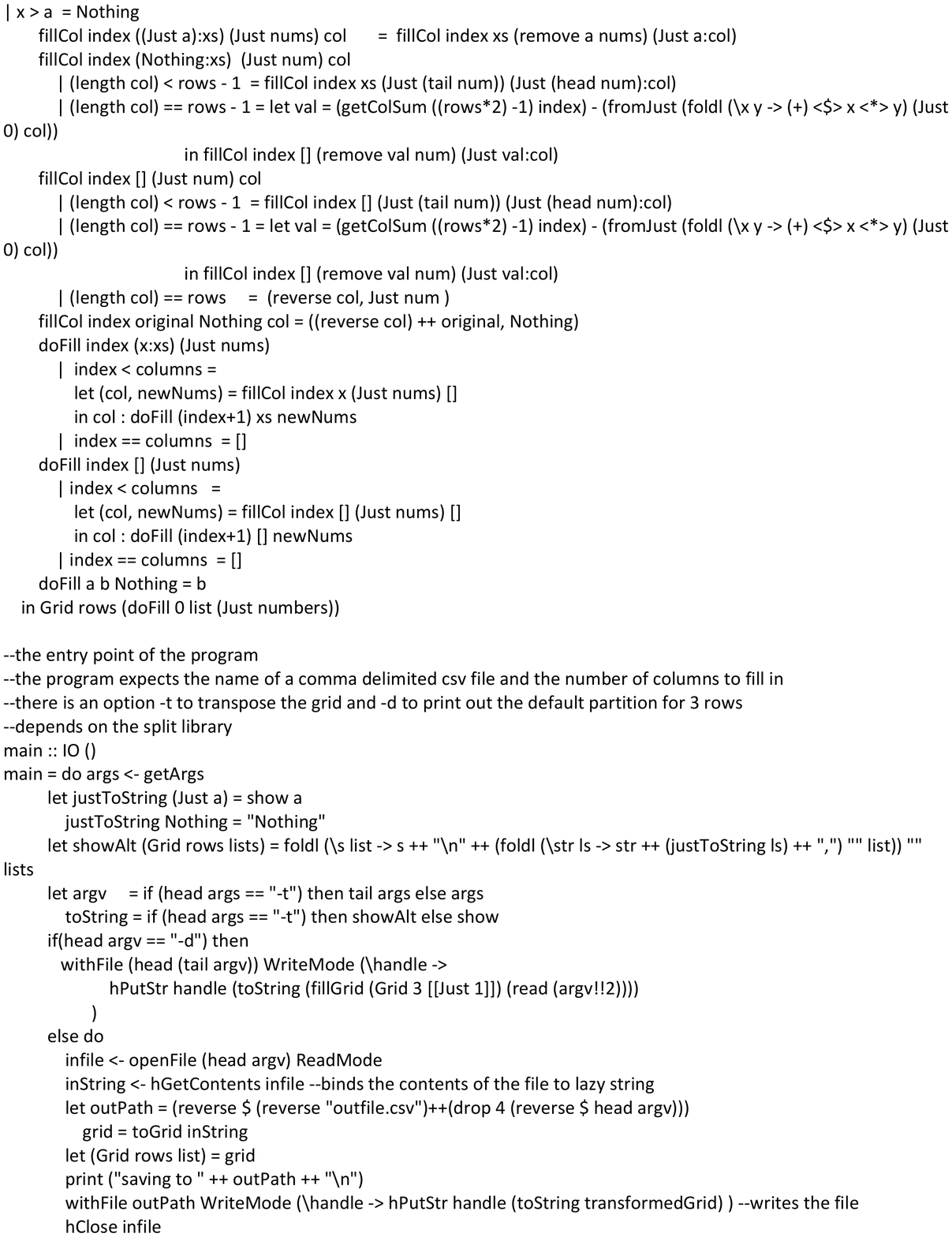}
  \end{center}
  \caption{\label{fig:hask} The Haskell program, part 2}
\end{figure}

\newpage

\begin{figure}[!htb]
  \begin{center}
     \includegraphics[width=150mm,scale=0.90]{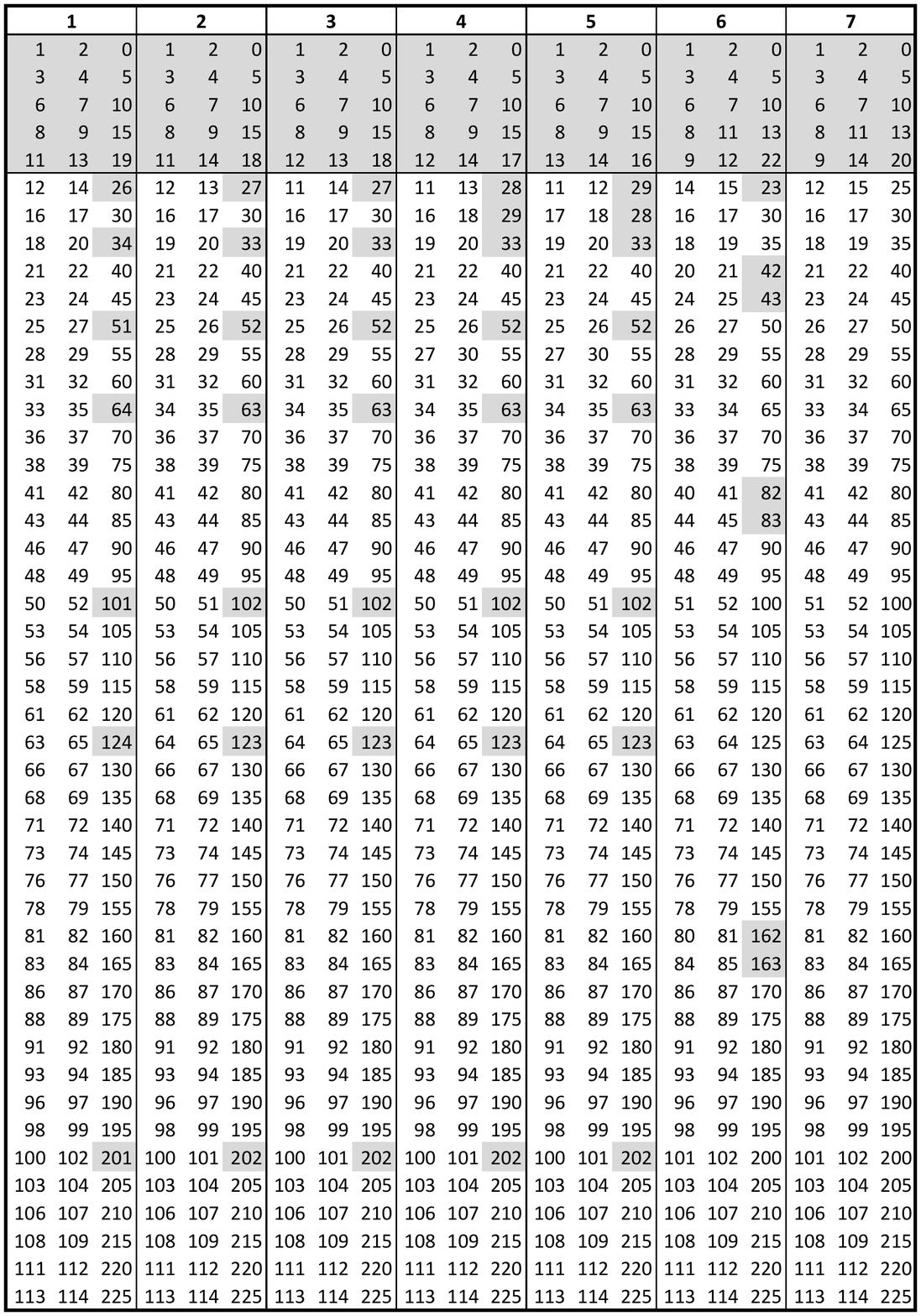}
  \end{center}
  \caption{\label{fig:part1} Partitions one through seven}
\end{figure}

\newpage

\begin{figure}[!htb]
  \begin{center}
     \includegraphics[width=150mm,scale=0.90]{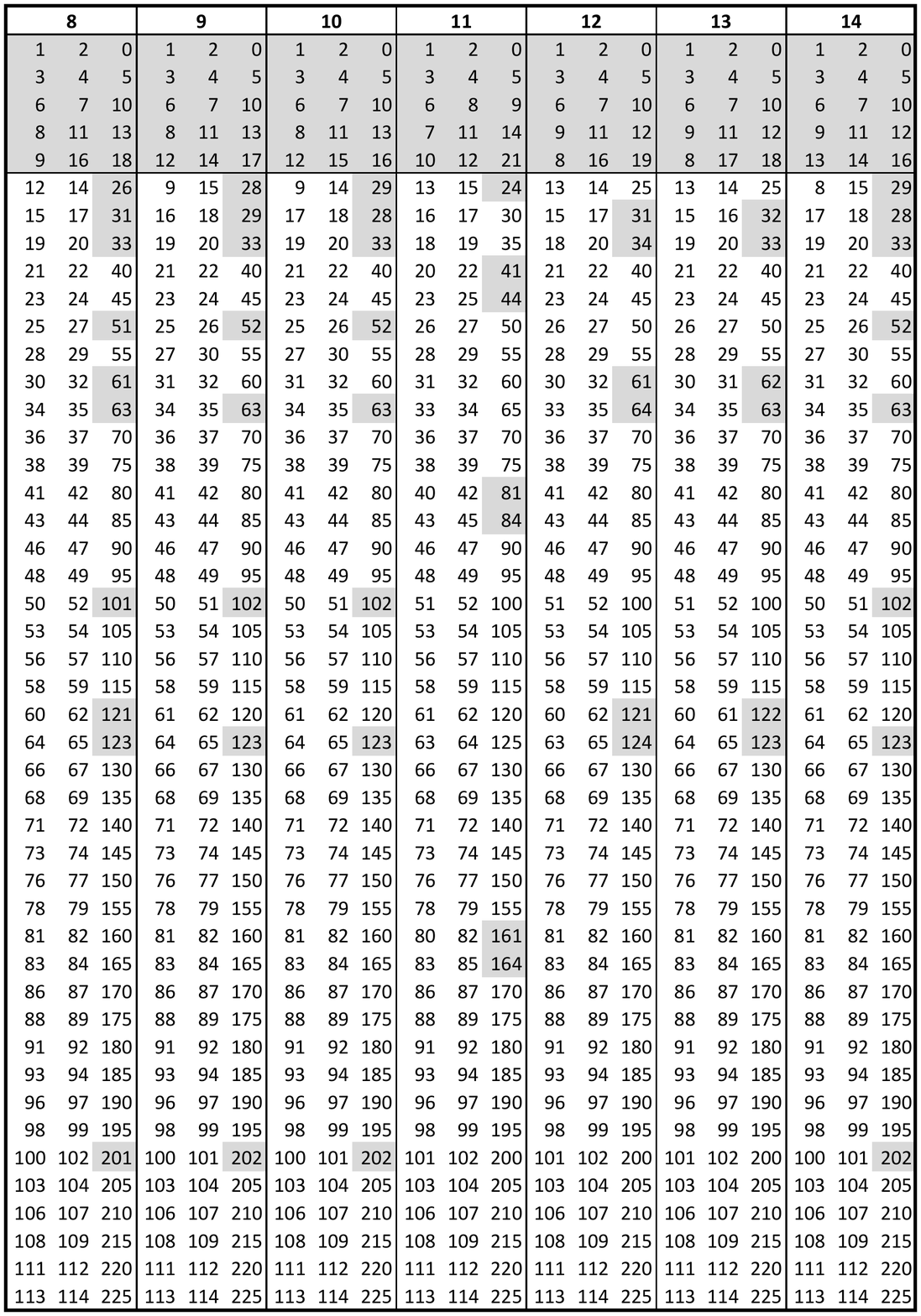}
  \end{center}
  \caption{\label{fig:part2} Partitions eight through fourteen}
\end{figure}

\newpage

\begin{figure}[!htb]
  \begin{center}
     \includegraphics[width=150mm,scale=2.50]{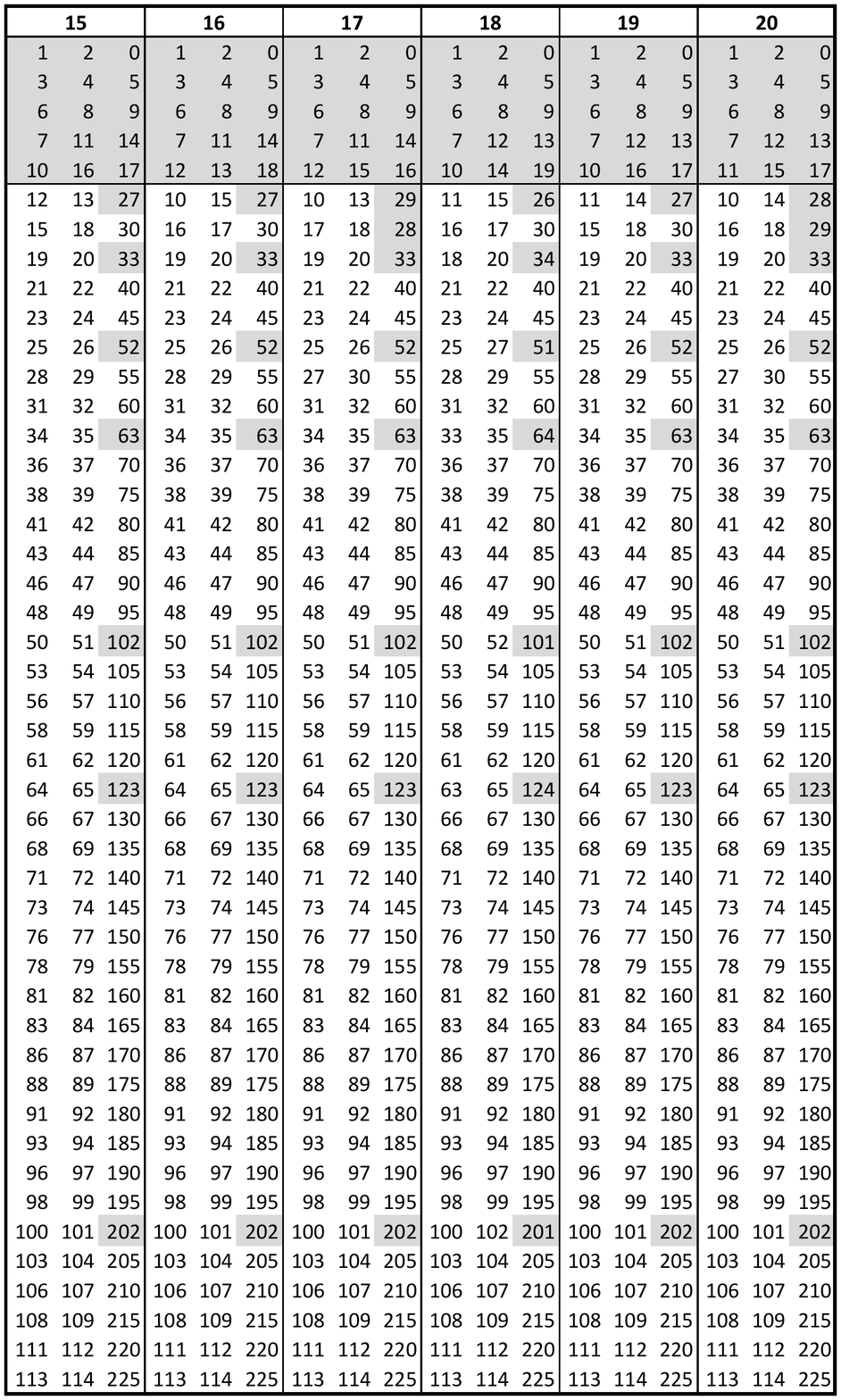}
  \end{center}
  \caption{\label{fig:part3} Partitions fifteen through twenty}
\end{figure}

\newpage

We wrote a Haskell program (shown in Figures \ref{fig:hask1} and \ref{fig:hask}) which applies the algorithm  from the proof of Proposition \ref{secsol} to the 20 options, and we list the first 46 elements in each set in Figures \ref{fig:part1}, \ref{fig:part2}, and \ref{fig:part3} (the sets $A$, $B$, and $C$ are listed vertically).

We see from the figures that we obtain 20 partitions, and only one of them (partition number 7 from Figure \ref{fig:part1}, which is actually the partition constructed in Proposition \ref{secsol} modulo a permutation of the first elements) seems to be equivalent to the standard partition. We will prove that this is actually true, and we count the number of equivalence classes in the following result.

\begin{proposition}\label{counteqcl}
The 20 partitions in Figures \ref{fig:part1}, \ref{fig:part2}, and \ref{fig:part3} contain eight equivalence classes:\\
Class 1: partitions 1 and 18.\\
Class 2: partitions 2, 3, 4, 5, 9, 10, 14, 15, 16, 17, 19, and 20.\\
Class 3: partition 6.\\
Class 4: partition 7 (it is equivalent to the standard partition).\\
Class 5: partition  8.\\
Class 6: partition 11.\\
Class 7: partition 12.\\
Class 8: partition 13.
\end{proposition}
\begin{proof}
We will describe a partition in each equivalence class by explaining how they differ from the standard partition.\\
Class 1: A partition in this class coincides with the standard partition after the 11-th elements, with the following exceptions:\\
On positions $2^k10+1$, $k\ge 0$, the standard partition has elements
$$2^k25+1, 2^k25+2,\mbox{ and }2^k50,$$
while a partition in Class 1 has elements
$$2^k25, 2^k25+2,\mbox{ and }2^k50+1.$$
We will only explain this case in detail, all the other ones are similar. When we use the algorithm from the proof of Proposition \ref{secsol} and we get to position $2^k20+1$, where we need to put in the numbers $2^{k+1}25+1=2^k50+1$ and $2^{k+1}25+2=2^k50+2$, we notice that the element $2^k50+1$ was already used on position $2^k10+1$, but $2^k50$ was not used. Therefore, we need to use $2^k50$ instead of $2^k50+1$, along with $2^k50+2$. The sum of the three elements on position $2^k20+1$ is
$$11(2^k20+1)-2\lfloor\frac{2^k20+1}{2}\rfloor -8=2^k200+3,$$
so the element in the third set is $2^k200+3-2^k50-2-2^k50=2^k100+1=2^{k+1}50+1$, so the claim is established by induction.\\[2mm]
\newpage

On positions $2^k12+2$, $k\ge 0$, the standard partition has elements
$$2^k30+3, 2^k30+4,\mbox{ and }2^k60+5,$$
while a partition in Class 1 has elements
$$2^k30+3, 2^k30+5,\mbox{ and }2^k60+4.$$
We also need to check that $2^k10+1\ne 2^l12+2$ for all $k,l\ge 0$. This is obvious because the left hand side is odd and the right hand side is even. The same argument holds for all equivalence classes except class 5.

 Finally, we note that the two types of exceptions are alternating, which follows from the following chain of inequalities which is true for all $k$:
$$2^k50+1<2^k60+5<2^k100+1<2^k120+4.$$
Class 2: A partition in this class coincides with the standard partition after the 11-th elements, with the following exceptions:\\
On positions $2^k10+1$, $k\ge 0$, the standard partition has elements
$$2^k25+1, 2^k25+2,\mbox{ and }2^k50,$$
while a partition in Class 2 has elements
$$2^k25, 2^k25+1,\mbox{ and }2^k50+2.$$
On positions $2^k12+2$, $k\ge 0$, the standard partition has elements
$$2^k30+3, 2^k30+4,\mbox{ and }2^k60+5,$$
while a partition in Class 2 has elements
$$2^k30+4, 2^k30+5,\mbox{ and }2^k60+3.$$
Class 3: Partition 6 coincides with the standard partition after the 9-th elements, with the following exceptions:\\
On positions $2^k+1$, $k\ge 3$, the standard partition has elements
$$2^{k-1}5+1, 2^{k-1}5+2,\mbox{ and }2^k5,$$
while partition 6 has elements
$$2^{k-1}5, 2^{k-1}5+1,\mbox{ and }2^k5+2.$$
\newpage

On positions $2^k+2$, $k\ge 3$, the standard partition has elements
$$2^{k-1}5+3, 2^{k-1}5+4,\mbox{ and }5(2^k+1),$$
while partition 6 has elements
$$2^{k-1}5+4, 5(2^{k-1}+1),\mbox{ and }2^k5+3.$$
Class 4: Partition 7 coincides with the standard partition starting with the 7-th elements.\\
Class 5: Partition 8 coincides with the standard partition after the 11-th elements, with the following exceptions:\\
On positions $2^k10+1$, $k\ge 0$, the standard partition has elements
$$2^k25+1, 2^k25+2,\mbox{ and }2^k50,$$
while partition 8 has elements
$$2^k25, 2^k25+2,\mbox{ and }2^k50+1.$$
On positions $2^k12+1$, $k\ge 0$, the standard partition has elements
$$2^k30+1, 2^k30+2,\mbox{ and }2^k60,$$
while partition 8 has elements
$$2^k30, 2^k30+2,\mbox{ and }2^k60+1.$$
On positions $2^k12+2$, $k\ge 0$, the standard partition has elements
$$2^k30+3, 2^k30+4,\mbox{ and }2^k60+5,$$
while partition 8 has elements
$$2^k30+4, 2^k30+5,\mbox{ and }2^k60+3.$$
We also need to check that $2^k10+1\ne 2^l12+1$ for all $k,l\ge 0$. This is true because the left hand side is congruent to 1 modulo 5, and the right hand side is not.\\[2mm]
Class 6: Partition 11 coincides with the standard partition after the 9-th elements, with the following exceptions:\\
On positions $2^k+1$, $k\ge 3$, the standard partition has elements
$$2^{k-1}5+1, 2^{k-1}5+2,\mbox{ and }2^k5,$$
while  partition 11 has elements
$$2^{k-1}5, 2^{k-1}5+2,\mbox{ and }2^k5+1.$$
On positions $2^k+2$, $k\ge 3$, the standard partition has elements
$$2^{k-1}5+3, 2^{k-1}5+4,\mbox{ and }5(2^k+1),$$
while partition 11 has elements
$$2^{k-1}5+3, 5(2^{k-1}+1),\mbox{ and }2^k5+4.$$
Class 7: Partition 12 coincides with the standard partition after the 11-th elements, with the following exceptions:\\
On positions $2^k12+1$, $k\ge 0$, the standard partition has elements
$$2^k30+1, 2^k30+2,\mbox{ and }2^k60,$$
while partition 12 has elements
$$2^k30, 2^k30+2,\mbox{ and }2^k60+1.$$
On positions $2^k12+2$, $k\ge 0$, the standard partition has elements
$$2^k30+3, 2^k30+4,\mbox{ and }2^k60+5,$$
while a partition in Class 1 has elements
$$2^k30+3, 2^k30+5,\mbox{ and }2^k60+4.$$
Class 8: Partition 13 coincides with the standard partition after the 13-th elements, with the following exceptions:\\
On positions $2^k12+1$, $k\ge 0$, the standard partition has elements
$$2^k30+1, 2^k30+2,\mbox{ and }2^k60,$$
while partition 13 has elements
$$2^k30, 2^k30+1,\mbox{ and }2^k60+2.$$
On positions $2^k12+2$, $k\ge 0$, the standard partition has elements
$$2^k30+3, 2^k30+4,\mbox{ and }2^k60+5,$$
while partition 13 has elements
$$2^k30+4, 2^k30+5,\mbox{ and }2^k60+3.$$
\end{proof}

Among the partitions  obtained from the 36 statements (different ways to fill in the first five elements in the three sets preserving the pattern of the sums) using the algorithm in the proof of Proposition \ref{secsol}, there are exactly five that are equivalent to the standard partition: the standard partition itself, the one obtained from the standard partition by switching 8 and 11 in the first set, and 15 in the third set with 12 in the second (this comes from number 15 in the list of 36 possible statements, and is equivalent with the standard partition if we take $N=5$ in Definition \ref{eqpar}), the one obtained from the standard partition by switching 7 in the second set with 8 in the first, and 10 in the third set with 9 in the second (this comes from number 19 in the list of 36 possible statements, and is equivalent with the standard partition if we take $N=4$ in Definition \ref{eqpar}), partition 7 in Figure \ref{fig:part1} (this comes from number 8 in the list, and is equivalent with the standard partition if we take $N=6$ in Definition \ref{eqpar}), and the partition obtained from the standard partition by swapping the following pairs, in order: 7 and 8, 9 and 10, 10 and 11, 14 and 15, 12 and 13 (this comes from number 26 in the list, and is equivalent with the standard partition if we take $N=6$ in Definition \ref{eqpar}). 

We end this section by investigating new ways to obtain partitions equivalent to the standard partition by starting with the standard partition and reshuffling elements such that the pattern of the sums is preserved. These are generalizations of the first two reshuffles of the standard partition described above (i.e. the ones coming from numbers 15 and 19 in the list of 36), and they can help create new partition problems which can be solved using the algorithm from the proof of Proposition \ref{secsol}, and whose solutions are equivalent to the standard partition. The way they work is to replace elements in blocks of two: replace two elements of a certain rank with sum $s$ with two elements of a different rank whose sum is also $s$. The result is the following:

\begin{proposition}\label{shuffle}
The following are infinitely many ways to obtain partitions equivalent to the standard partition by starting with the standard partition and reshuffling elements such that the pattern of the sums is preserved:\\
i) Replace for each $k\ge 1$ the elements of rank $4k$ in the first and last set by the elements of rank $6k-1$ in the first two sets.\\
ii) Replace for each $k\ge 0$ the elements of rank $4k+3$ in the last two sets by the elements of rank $6k+4$ in the first two sets.
\end{proposition}
\begin{proof}
i) We use Definition \ref{stanpar}, and we compute the sum of the elements of rank $4k$ in the first and last set of the standard partition. We get $8\cdot 4k-2k-2-5=30k-7$. The sum of the elements of rank $6k-1$ in the first two sets is the same: $6(6k-1)-2(3k-1)-3=30k-7$.\\
ii) If $k\ge 0$, the elements of rank $4k+3$ in the last two sets in the standard partition have sum $8(4k+3)-2k-1-6=30k+17$, while the elements of rank $6k+4$ in the first two sets have sum $6(6k+4)-2(3k+2)-3=30k+17$.
\end{proof}
\section{A generalization}
In this section we generalize the problem studied in the first two section and we assume that $m=2t+1$ is an odd number.
\begin{definition}\label{mstanpar} The {\em\bf $m$-standard partition} of the non-negative integers is given by
\begin{align*}
S_1 &=\{(t+1)(n-1)-\lfloor\frac{n}{2}\rfloor+1\mid n\ge 1\}\\
S_2 & =\{(t+1)(n-1)-\lfloor\frac{n}{2}\rfloor+2\mid n\ge 1\}\\
&\vdots\\
S_t & =\{(t+1)(n-1)-\lfloor\frac{n}{2}\rfloor+t\mid n\ge 1\}\\
S_{t+1} &= \{(2t+1)(n-1)\mid n\ge 1\}
\end{align*}
Note that the sum of the $n$-th elements is $S(n)=(t+1)^2(n-1)+t\lfloor\frac{n-1}{2}\rfloor +\frac{t(t+1)}{2}$ (because $n-1=\lfloor\frac{n}{2}\rfloor+\lfloor\frac{n-1}{2}\rfloor$). It is easy to see that $S_1$ consists of the numbers with remainder 1 or $t+1$ (mod $m$), $\ldots$
the elements of $S_t$ are congruent to $t$ or $2t$ (mod $m$), while $S_{t+1}$ is the set of
multiples of $m=2t+1$.
\end{definition}

Equivalent partitions are defined similar to Definition \ref{eqpar}. We can now construct more partition problems similar to the problem studied in the first two sections. In order to make sure that the solution is equivalent to the $m$-standard partition, when we design the problem we only use elements from the first five columns of the $m$-standard partition. The following example is a problem in the case $m=7$.\\[2mm]
{\bf Problem} The non-negative integers are divided into four groups as follows:
\[ \begin{array}{ccccccccc}
A & = & \{  1, & 6, & 9, & 21, & 28, &\ldots  \}\\
B & = & \{  0, & 5, & 8, & 14, & 15, &\ldots  \}\\
C & = & \{  3, & 7, & 13, & 10, & 17, &\ldots  \}.\\
D & = & \{  2, & 4, & 11, & 12,   & &\ldots  \}.\\
\end{array} \]
Explain.\\[2mm]
{\bf Solution.} The sum of the first elements is 6. The sum of the second elements is 22, which is 16 more than the first sum. The sum of the third elements is 41, which is 19 more than the second sum. The sum of the fourth elements is 57, whch is again 16 more than the third sum. Asssuming that the fifth sum is $57+19=76$, we find that the fifth element of $D$ is $76-28-15-17=16$. Now we notice that the set of the first five elements in the four  sets is $\{0,1,2,3,4,5,6,7,8,9,10,11,12,13,14,15,16,17,21,28\}.$ Therefore, we can set the fifth elements of the first thre sets to be 18, 19, and 20, and we find that the fifth element of $D$ is $76+16-18-19-20=35$. We continue as above with the first unused numbers (i.e. 22, 23, and 24) as the sixth elements of the first three sets, and the next multiple of 7 (i.e. 42) as the sixth element of $D$. The solution is a partition equivalent to the $7$-standard partition, because after the sixth column we get exactly the columns of the $7$-standard partition.

\begin{remark}
We designed the statement of the problem as follows: we started with the 7-standard partition, we permuted the elements of the same rank, and we swapped the pairs of numbers 13 and 14, and 11 and 10 (we can do this because $10+14=11+13$).
\end{remark}

We will not give any details, but results similar to those in Section \ref{sect2} may be given for generalized partitions. For the case $m=5$ we saw in Proposition \ref{counteqcl} that there are 8 equivalence classes in the partitions obtained from all possible ways of choosing the first five elements of the sets in the partition. We used a computer to determine that for $m=7$ there are 13 equivalence classes, for $m=9$ there are 19, and for $m=11$ the number of equivalence classes is 26.

\section*{Acknowledgments} 

We thank Constantin Manoil for helpful conversations.
 

{\footnotesize  
\medskip
\medskip
\vspace*{1mm} 
 
\noindent {\it Eunice Krinsky}\\  
California State University, Dominguez Hills \\
Mathematics Department\\
1000 E Victoria St \\
Carson, CA 90747\\
E-mail: {\tt ekrinsky@csudh.edu}\\ \\  

\noindent {\it Serban Raianu}\\  
California State University, Dominguez Hills \\
Mathematics Department\\
1000 E Victoria St \\
Carson, CA 90747\\
E-mail: {\tt sraianu@csudh.edu}\\ \\

\noindent {\it Alexander Wittmond}\\  
California State University, Dominguez Hills \\
1000 E Victoria St \\
Carson, CA 90747\\
E-mail: {\tt awittmond1@toromail.csudh.edu} }\\ \\   

\vspace*{1mm}\noindent\footnotesize{\date{ {\bf Received}: January 1, 2017\;\; {\bf Accepted}: January 2, 2017}\\
\vspace*{1mm}\noindent\footnotesize{\date{ {\bf Communicated by Some Editor}}}

\end{document}